\documentclass[11pt]{article}

\usepackage{mathtools}
\usepackage{amsmath, amsthm, amsfonts,amssymb}
\usepackage{enumitem}
\usepackage{graphicx}
\usepackage{colortbl}
\usepackage{tikz}
\usepackage[utf8]{inputenc}
\usepackage{esint}
\usepackage{mathrsfs}
\usepackage{subfig,float}
\usepackage[T1]{fontenc}
\usepackage{mathrsfs}  
\usepackage{bbm} 
\usepackage{enumitem}
\usepackage{mathtools}
\usepackage{dsfont}
\usepackage{ushort}
\usepackage{accents}

\usepackage{tikz}
\usetikzlibrary{angles, quotes,calc,patterns}

\DeclareMathOperator{\supp}{supp}

\newcommand\blfootnote[1]{%
	\begingroup
	\renewcommand\thefootnote{}\footnote{#1}%
	\addtocounter{footnote}{-1}%
	\endgroup
}

\def\d{\,\mathrm{d}}

\newcommand{\sign}{\text{sgn}}
\DeclareMathOperator{\hess}{Hess}

\usepackage[colorlinks=true,linkcolor=blue,citecolor=blue,urlcolor=blue,breaklinks]{hyperref}

\usepackage{hyperref}
\usepackage{breakurl}
\usepackage[square,sort,comma,numbers]{natbib}
\usepackage{url}
\nocite{}

\usepackage[square,sort,comma,numbers]{natbib}
\definecolor{lpink}{rgb}{0.96, 0.76, 0.76}
\definecolor{dpink}{rgb}{0.97, 0.51, 0.47}
\definecolor{sky}{rgb}{0.53, 0.81, 0.92}
\definecolor{salmon}{rgb}{1.0, 0.55, 0.41}
\definecolor{orman}{rgb}{0.24, 0.7, 0.44}
\definecolor{aciksari}{rgb}{0.91, 0.84, 0.42}
\definecolor{dgrey}{rgb}{0.52, 0.52, 0.51}

\def\R{\mathbb{R}}

\def\S{\mathbb{S}}

\def\1{\mathds{1}}

\let\mc=\mathcal

\def\d{\,\mathrm{d}}

\def\p{\,\partial}

\newtheorem{thm}{Theorem}[section]

\newtheorem{lem}[thm]{Lemma}
\newtheorem{prp}[thm]{Proposition}

\theoremstyle{definition}

\theoremstyle{remark}
\newtheorem{remark}[thm]{Remark}

\newtheorem{hypothesis}{Hypothesis}

\hypersetup{pdftitle={RunAndTumbleHarris}}
\hypersetup{pdfauthor={Josephine Evans  \& Havva Yoldaş }}

\author{Josephine Evans\footnote{Warwick Mathematics Institute, University of Warwick, Zeeman Building, Coventry CV4 7AL, United Kingdom. josephine.evans@warwick.ac.uk} 
	\and Havva Yolda\c{s} \footnote{Delft Institute of Applied Mathematics, Faculty of Electrical Engineering, Mathematics and Computer Science, Delft University of Technology, Mekelweg 4, 2628CD Delft, The Netherlands. h.yoldas@tudelft.nl}}

\title{Trend to equilibrium for run and tumble equations with non-uniform tumbling kernels}

\makeindex

\begin{document}
	
	\maketitle
	
	\vspace{-10pt}

	\begin{abstract}
		\noindent 
		We study the long-time behaviour of a run and tumble model which is a kinetic-transport equation describing bacterial movement under the effect of a chemical stimulus. The experiments suggest that the non-uniform tumbling kernels are physically relevant ones as opposed to the uniform tumbling kernel which is widely considered in the literature to reduce the complexity of the mathematical analysis. We consider two cases: (i) the tumbling kernel depends on the angle between pre- and post-tumbling velocities, (ii) the velocity space is unbounded and the post-tumbling velocities follow the Maxwellian velocity distribution. We prove that the probability density distribution of bacteria converges to an equilibrium distribution with explicit (exponential for (i) and algebraic for (ii)) convergence rates, for any probability measure initial data. To the best of our knowledge, our results are the first results concerning the long-time behaviour of run and tumble equations with non-uniform tumbling kernels.
		
		
	\end{abstract}
	
	\tableofcontents
	
	\section{Introduction} \label{sec:intro}
	
	We consider a kinetic-transport equation which describes the movement of biological microorganisms biased towards a chemoattractant. The model is called the \emph{run and tumble} equation and introduced in \cite {A80, S74} based on some experimental observations \cite{BB72} on the bacterium called \emph{Escherichia coli (E. coli)}. The equation is given by 
	\begin{align} \label{eq:rt}
		\begin{split}
			\partial_t f + v\cdot \nabla_x f &= \int_{\mathcal{V}} \lambda (v' \cdot \nabla_x M(x)) \kappa (v, v') f(t,x, v') \d v' - \lambda (v \cdot \nabla_x M(x)) f(t,x,v) \\
			f(0,x,v) &= f_0(x,v) 
		\end{split}
	\end{align} 
	where $f:=f (t,x,v) \geq 0$ is the density distribution of microorganisms at time $t \geq 0$ at a position $x \in \mathbb{R}^d$, moving with a velocity $v \in \mathcal{V} \subseteq \R^d$. 
	
	The term $\lambda (v' \cdot \nabla_x M(x)) \kappa (v, v')$ is called the \emph{tumbling frequency} where $\lambda : \mathbb{R} \to [0, \infty) $ is the \emph{tumbling rate}. The tumbling rate $\lambda$ depends on the gradient of the external signal $M$ along the direction of the velocity $v$, and it is defined by
	\begin{align*} 
		M (x)= m_0 + \log (S(x)),
	\end{align*} where $m_0 \in \R^+_0$ represents the external signal in the absence of a chemical stimulus and $S(x)$ is a given function for the density of the chemoattractant. 
	In Eq. \eqref{eq:rt}, the \emph{tumbling or turning kernel} $\kappa(v,v') $ is a probability distribution on the space $\mathcal{V}$ and gives the probability of moving from velocity $v$ to velocity $v'$, i.e. $\int_{\mc V} \kappa (v,v') \d v' = 1$. 
	
	In the case of peritrichous bacteria such as \emph{E. coli} and \emph{Salmonella typhimurium}, experiments conducted in \cite{BB72,M80} suggest that $\kappa$ depends only on the relative angle $\theta$ between the pre- and post-tumbling velocities $v$ and $v'$ respectively. Particularly, for bacterium \emph{E. coli}, the tumbling kernel $\kappa$ is given by 
	\begin{align*}
		\kappa(v, v') = \kappa (\theta) = \frac{g(\theta)}{2 \pi \sin \theta}, \quad  \theta = \arccos \left(\frac{v \cdot v'}{|v||v'|}\right), 
	\end{align*} where $g(\theta)$ is a sixth order polynomial satisfying $g(0) = g(\pi) =0$ (see \cite{OH02, BFSC96} for more details). The exact form of $g$ is provided in \cite{FFC93} by polynomial fitting to the data of \cite{BB72}. Here we will work in a bounded velocity space.
	
	In our previous paper \cite{EY21} we studied this equation under the assumption that $\kappa$ was uniformly bounded above and below. However, this assumption is not realistic as the bacteria are not able to turn a full half circle. We aim to extend our previous work to the case where the maximum turning angle of the bacteria may be bounded.
	
	We are also interested in unbounded velocity spaces, i.e., $v \in \mc V  =\R^d$. In this setting, we consider that the tumbling kernel is given by the Maxwellian distribution on the post-tumbling velocities independently from the pre-tumbling velocities, i.e.,
		\begin{align*}
			\kappa (v,v') = \mc M (v) = \frac{1}{(2\pi)^\frac{d}{2}} e^{\frac{-|v|^2}{2}}. 
	\end{align*}
	To the best of our knowledge, we provide the first results concerning the long-time behaviour of the run and tumble equation with these non-uniform tumbling kernels which are more physically relevant in terms of modelling the chemotactic bacterial motion. We believe the main reason for this is the fact that the classical hypocoercivity techniques such as \cite{DMS15} cannot be used for the run and tumble equation even though it is a linear, conservative kinetic equation. On the other hand, Harris-type theorems (see e.g. \cite{CM21, Y22}) proved very effective for obtaining quantitative hypocoercivity results, especially for kinetic equations arising from applied sciences where classical techniques provide limited results. We elaborated on this fact in our previous paper \cite{EY21} in detail and a brief explanation can be found below in the paragraph \emph{Motivation and novelty}.
	\paragraph{Summary of previous results.} Previous important works on the linear run and tumble equation include \cite{OH00, OH02, CRS15, MW17, C19}. In \cite{OH00,OH02}, the authors study the diffusion approximation to a linear run and tumble equation and the diffusion limit of this equation to obtain macroscopic chemotaxis equations, respectively. Using the $L^2$ hypocoercivity techniques developed in \cite{DMS15}, the authors \cite{CRS15} show the existence of a unique equilibrium and exponential decay towards it in dimension $d=1$. This paper works with the assumption that the tumbling rate $\lambda$ can take two values depending if the bacteria is travelling up or down the gradient of the chemoattractant density. Though it is expressed differently in \cite{CRS15} the tumbling rate can be written as
	\begin{align*}  
		\lambda (x,v,v') = 1 + \chi \sign (x \cdot v), \quad  \chi \in (0,1),
	\end{align*} where $\chi$ is called the chemotactic sensitivity. In  \cite{MW17}, the authors extended this results to higher dimensions $d\geq1$ considering
	\begin{align*} 
		\lambda  (x,v,v') = 1 -\chi \sign (\p_t S + v \cdot \nabla_x S), \quad \chi  \in  (0,1). 
	\end{align*} The result in \cite{MW17} works under the assumption that the concentration of the chemoattractant $S(x)$ is radially symmetric and decreasing in $x$ such that $S(x) \to 0$ as $|x| \to \infty$. In our previous paper, \cite{EY21} we improve the result in \cite{MW17} by proving the exponential convergence to unique equilibrium without requiring $S(x)$ to be radially symmetric. Our result is valid in an arbitrary dimension $d \geq 1$ where the tumbling kernel $\lambda$ can take a much more general form. Most importantly, we show that existence and convergence to a steady state hold when $\lambda$ is a Lipschitz function. Our techniques in \cite{EY21} are based on Harris-type theorems coming from the ergodic theory of Markov processes. They differ from the techniques used in \cite{MW17} that are based on the Krein-Rutman theorem. Moreover in \cite{EY21}, we also consider a non-linear run and tumble model where we can prove the exponential decay to a unique equilibrium. In \cite{C19}, the author studies a non-linear coupled run and tumble equation in one dimension $d=1$. Even though \cite{C19} is mainly concerned with the travelling wave solutions of the non-linear equation, as an intermediate step, the author shows the existence of steady states for the linear equation.
	
	All these previous works deal with situations where the tumbling kernel, which tells us how the post-tumbling velocities depend on the pre-tumbling velocities, is essentially uniform, i.e., $\kappa \equiv 1$, and on a bounded velocity space. The main goal of this work is to look at physically more realistic cases where the tumbling kernel is not uniform and the bacteria are not able to turn to every angle in one tumbling event.

	\paragraph{Motivation and novelty.}
	In our previous paper \cite{EY21} we studied Eq. \eqref{eq:rt} in the case where the velocity space was a ball and the post-tumbling velocities were uniformly distributed so that $\kappa \equiv 1$. We showed exponential convergence to equilibrium in suitable weighted total variation distances. The result was built around probabilistic techniques called Harris-type theorems coming from the theory of Markov processes. The goal of this paper is to extend the previous result to a wider class of tumbling kernels that are non-uniform and physically more relevant for modelling the motion of chemotactic bacteria.
	
	The typical tools for showing convergence to equilibrium for kinetic equations come from \emph{hypocoercivity}. It is important to note that "classical" hypocoercivity techniques such as \cite{V09, DMS15} cannot be applied without strong a priori knowledge of the steady state which we do not have for the run and tumble equation. This is the main reason for scarce results on the long-time behaviour with arbitrary dimensions $d \geq 1$ even though the run and tumble equation is widely studied. Harris-type theorems (see, e.g., \cite{CM21,H16,HM11} and references therein) provide the existence and uniqueness of the steady state as a by-product while simultaneously showing convergence to equilibrium. Using these tools, one can obtain quantitative hypocoercivity results in weighted total variation (or in weighted $L^1$) distances independently from the initial data (see, e.g., \cite{CCEY20, Y22}).
	
	Harris-type theorems are based on verifying two hypotheses: minorisation and geometric drift conditions. The minorisation condition requires providing a quantifiable positive lower bound on the process. Therefore, non-uniform tumbling kernels pose additional challenges compared to the case with a uniform tumbling kernel when using Harris-type theorems. This is because we need to prove lower bounds on the law uniform over a large set of initial conditions. This means that we need to find some possible paths a bacterium can take when moving from one point in phase space to another. The fact that the tumbling angle is bounded means that bacteria may only be able to travel between two points along paths involving many tumbling events. Tracking bounds on the probabilities of these complex paths (and paths close to them) is challenging and required us to generate new technical tools. 
	
	\subsection{Assumptions and main results} \label{sec:assumptions}
	The two main results of our paper concern angularly dependent tumbling kernels and unbounded velocity spaces. In both cases, we will make the following assumptions on the tumbling rate, $\lambda$, and the logarithm of the chemoattractant concentration, $M$.
	
	\begin{itemize}
		\item[\bf(H1)] \label{hyp:lambda}
		The tumbling rate $\lambda(m): \mathbb{R} \to (0, \infty)$ is a function of the form
		\begin{align} \label{asm:lambda} \tag{$\lambda$}
			\lambda(m) = 1- \chi\psi(m),  \quad \chi \in (0,1)
		\end{align} where $\psi$ is a bounded (with $\|\psi\|_\infty \leq 1$), odd, increasing function and $m\psi(m) \in W^{1, \infty}_{loc}(\mathbb{R})$. 
		\item[\bf(H2)] \label{hyp:int_bound}
		There exists a contstant a strictly positive integer $b>0$, such that for every $B>0$, there exists $c > 0$ depending on $B$ so that 
		\begin{align}\label{asm:m_psi_m}  \tag{$m$}
			m \psi(m) \geq c|m|^b  
		\end{align} for $|m| \leq B$. We note that this holds if $\psi$ is the sign function or if it is odd and differentiable around zero with strictly positive $k^{th}$ derivative for some $k \geq 1$. 
		\item[\bf(H3)] \label{hyp:M}
		We suppose that  $M(x) \to-\infty$  as $|x| \to \infty$, $|\nabla_x M(x)|$ is bounded and that there exist $R\geq 0$ and $m_* >0$ such that whenever $|x|>R$ we have
		\begin{align} \label{asm:M} \tag{$M$}
			|\nabla_x M(x)| \geq m_*.
		\end{align} 
		Moreover, we suppose that $\hess (M)(x) \to 0$ as $|x|\to  \infty$ and  $|\hess(M) (x)|$ is bounded.
	\end{itemize}
	
	The following assumptions concern the tumbling kernels we work with. 
	\begin{itemize}
		\item[\bf(H4)] \label{hyp:K_1} We assume that $\mathcal{V} = V_0 \mathbb{S}^{d-1}$ and there exist $\alpha, \beta >0$ such that
		\begin{align*}\label{asm:K_1} \tag{$\kappa_1$}
			\kappa(v,v') = \kappa_1(\theta) \geq \beta \mathds{1}_{|\theta| < \alpha} (\theta),
		\end{align*} where  $\theta = \arccos \left( \dfrac{v \cdot v'}{V_0^2} \right )$  and $\kappa_1$ is a decreasing function of $|\theta|$ (similar arguments work if $\kappa_1$ is even and bounded below by a decreasing function of $|\theta|$).
		\item[\bf{(H5)}] 	\label{hyp:K_2} We assume that $\mathcal{V} = \R^d$ and the tumbling kernel is given by the Maxwellian distribution on the post-tumbling velocities independently from the pre-tumbling velocities, 
			\begin{align} \tag{$\kappa_2$}
				\label{eq:K_Maxwell}
				\kappa(v, v') = \kappa_2  (v) = \frac{1}{(2 \pi)^{\frac{d}{2}}}e^{-\frac{|v|^2}{2}}.
		\end{align}
		
		After the assumptions, we state the main results of the paper below.
	\end{itemize}
	\begin{thm}[Angularly dependent tumbling kernel \eqref{asm:K_1}] 
		\label{thm:angulardep}
		Suppose that $ t \mapsto f_t $ is the solution of Equation \eqref{eq:rt} with initial data $f_0 \in \mathcal{P} (\R^2 \times \mathbb{S}^1 )$. We suppose that hypotheses 
		\textbf{\em (H1), (H2), (H3)} and \textbf{\em (H4)} are satisfied. Then there exist positive constants $C, \sigma$ (independent of $f_0$) such that 
		\begin{align} \label{eq:thm1}
			\|f_t - f_{\infty} \|_* \leq C e^{-\sigma t} \| f_0 -f_{\infty}\|_*, 
		\end{align} where $f_{\infty}$ is the unique steady state solution to Equation \eqref{eq:rt}. The norm $\|\cdot\|_*$ is the weighted total variation norm defined by 
		\begin{align} \label{eq:norm1}
			\|\mu\|_*  := \int_{\R^d} \int_{\S^{d-1}}  \Big ( 1 - \frac{\gamma}{1-C_\kappa} v\cdot \nabla_x M(x) - \gamma A v\cdot \nabla_x M (x)\psi(v \cdot \nabla_x M(x)) \Big )e^{-\gamma M(x)}  |\mu| \d v \d x,
		\end{align} where $\gamma, A, C_\kappa > 0$ are positive constants that can be computed explicitly and will be chosen so that $\|\cdot\|_*$ is indeed a norm, and $\mu$ is a finite measure.
	\end{thm}
	\begin{remark} We believe that Theorem \ref{thm:angulardep} works in arbitrary dimension $d \geq 1$. The reason for stating Theorem \ref{thm:angulardep} in dimension $d=2$ is that we provide the proof of Proposition \ref{prop:minorisation_angle_dependent} only in $d=2$. We do not believe that there is a major mathematical obstacle in proving it in higher dimensions. However, even in $d=2$, the computations become delicate, and the notations get intricate. Therefore, we decided to provide it in $d=2$ to keep the exposition of our ideas clear.
	\end{remark}
	
	\begin{thm}[Unbounded velocity space with tumbling kernel \eqref{eq:K_Maxwell}] \label{thm:unbounded}
		Suppose that $ t \mapsto f_t$ is the solution of Equation \eqref{eq:rt} with initial data $f_0 \in \mathcal{P}( \R^d \times \R^d)$. We suppose that hypotheses \textbf{\em (H1), (H2), (H3)} and \textbf{\em (H5)} are satisfied and assume further that $M \hess(M)$ is bounded. Then there exists a positive constant $C>0$ such that 
		\begin{align}\label{eq:thm2}
			\|f_t - f_\infty\|_{TV} \leq C t^{-1} M_{f_0}
		\end{align} where 
			\begin{align} \label{eq:norm2}
				M_{f_0} :=  \int_{\R^d} \int_{\R^d} f_0(x,v) \Big (1 + M(x)^2 + 2  v \cdot \nabla_x M(x) M (x)\Big ( 1+ \frac{\chi}{1+\chi}\psi \left ( v \cdot \nabla_x M (x) \Big  )\Big ) + Av^2 \right ) \d v \d x,
			\end{align} with $A >0$ is a constant that can be computed explicitly and it is sufficiently large so that $M_{f_0}>0$.
	\end{thm}
	
	Even though we study the long-time behaviour of the run and tumble equation \eqref{eq:rt} with the non-uniform tumbling kernels in this paper, we would like to briefly comment on the Cauchy theory for these equations. 
	\paragraph{Cauchy theory for Equation\eqref{eq:rt}.}  As Eq. \eqref{eq:rt} is a linear integro-differential equation with bounded coefficients, showing the existence and uniqueness of  global-in-time, measure-valued solutions is relatively standard. One could either use Picard iteration arguments to construct short-time solutions and then use the fact that $\lambda$ is bounded to show these can be glued together globally in time or one can directly write down a Markov process whose law satisfies Eq. \eqref{eq:rt}. We briefly explain how to do the latter. Let us generate a Poisson process with intensity $(1+\chi)$ and call its jump times $J_1, J_2, \dots$ and a series of thinning variables $U_1, U_2, \dots$ independent and all having the uniform law on $[0, 1+ \chi]$ then we define initial points $(X_0, V_0)$ having law $f_0$ then set $J_0 = 0$ and for $t \in (J_i, J_{i+1})$ we write $X_t = X_{J_i} + (t-J_i)V_{J_i}$ and $V_t = V_{J_i}$ then for $t = J_{i+1}$ we set $X_t = X_{J_i} + (J_{i+1}- J_i)V_{J_i}$ then if $U_{i+1} \leq \lambda(V_{J_i}\nabla_x M(X_t))$ we generate $V_{J_{i+1}}$ as a new random variable having law $\kappa(V_{J_i}, \cdot)$ and if $U_{i+1} > \lambda(V_{J_i} \nabla_x M(X_t))$ we set $V_{J_{i+1}} = V_{J_i}$. 
	
	\paragraph{Plan of the paper.} After listing our assumptions and stating our main results in Section \ref{sec:assumptions}, in the following section (Section \ref{sec:harris}), we describe our methodology, particularly we state Harris's theorem in geometric and subgeometric settings, Theorems \ref{thm:harris} and \ref{thm:subgeo} respectively. Then,  Sections \ref{sec:angle-dependent} and \ref{sec:unbounded-velocities} are dedicated to proving Theorems \ref{thm:angulardep} and \ref{thm:unbounded} for the cases (i) angularly dependent tumbling kernels and (ii) unbounded velocity spaces respectively.

	\subsection{Methodology}  \label{sec:harris}
	This paper is an extension of the work in \cite{EY21} and as such is built on Harris's theorem from Markov process theory. More precisely, Harris-type theorems are used in showing geometric (exponential) or sub-geometric (algebraic) rates of convergence to a unique equilibrium solution for equations that can be described as Markov processes. Harris-type theorems rely on verifying two hypotheses: a Foster-Lyapunov condition and a uniform minorisation condition. 
	
	We briefly introduce the functional setting and some notations below and then we provide the statements of the theorems both in geometric and sub-geometric settings in the spirit of \cite{HM11,MS16,CM21} and the references therein.  We skip the proofs of these theorems.
	
	\paragraph{Notations.} We consider a measurable space $(\Omega, \Sigma)$ where $\Omega$ is a Polish space endowed with a probability measure. We denote the space of probability measures on $\Omega$ by $\mc P (\Omega)$. Note that in our setting $\Omega = \R^d \times \mc V$ so that $(x, v) \in \R^d \times \mc V = \Omega$. We sometimes use the notation $z := (x,v)$. 
	
	We define the weighted total variation (or weighted $L^1$) distance by
	\begin{align*}
		\|\mu\|_\phi :=  \int_\Omega \phi(z) |\mu| (\d z), 
	\end{align*} where $\mu$ is a finite measure or a measurable function and $\phi: \Omega \to [1, +\infty) $ is a measurable weight function. 
	
	We call $(S_t)_{t \geq 0}$ a Markov (or stochastic) semigroup if it is a linear semigroup conserving mass and positivity. Remark that if $f$ solves Eq. \eqref{eq:rt}, then $f(t,x,v) = S_tf_0(x,v)$ and  $(S_t)_{t \geq 0}$ is a Markov semigroup since Eq. \eqref{eq:rt} is positivity and mass-preserving.
	Moreover, we also use the notation $\p_t f = \mc L [f]$ equivalently to Eq. \eqref{eq:rt}. If Eq. \eqref{eq:rt} has a stationary soluton $f_\infty$, this means that $\p_t f_\infty =\mc L [f_\infty] = 0 $ and $f_\infty$ is an invariant measure for the semigroup $S_t$, i.e., $S_t f_\infty = f_\infty$.

	\begin{thm}[Harris's Theorem] \label{thm:harris}
		Let $(S_t)_{t \geq 0}$ be a Markov semigroup satisfying the following hypotheses:
		\begin{hypothesis} [Foster-Lyapunov condition]\label{hyp:FL1} There exist positive constants $\zeta, D$ and a continuous, measurable function $\phi : \Omega \to [1, +\infty)$ such that 
			\begin{align}
				\label{FL1} \tag{FL1}
				\mc L^* [\phi] (z) \leq D -\zeta \phi(z).
			\end{align}
		\end{hypothesis}
		\begin{hypothesis}[Minorisation condition]\label{hyp:M1}
			There exist a probability measure $\nu$, a constant
			$\beta \in (0,1)$ and some time $\tau >0 $ such that
			\begin{align}
				\label{M1} \tag{M1}
				S_\tau \delta_z \geq \beta \nu, \quad \mbox{for all } z \in \mathcal C,
			\end{align} where $\mathcal{C} =  \{ z \colon \phi(z) \leq R \}$ for some $R >  \frac{2D (1-e^{-\zeta \tau})}{\zeta (1-\alpha)}$.
		\end{hypothesis}
		Then $(S_t)_{t \geq 0}$ has a unique invariant measure $\mu_\infty$ and for any $\mu \in \mc P (\Omega)$ there exist some constants $C>1, \sigma >0$ such that for all $t \geq 0$ we have
		\begin{align}
			\|S_t \mu - \mu_\infty\|_\phi \leq C e^{-\sigma t} \|\mu - \mu_\infty\|_\phi.
		\end{align}
	\end{thm}
	\begin{remark}
		The constants $C, \sigma $ can be computed explicitly in terms of $D, \tau, \alpha, \beta, \zeta$ (see Remark 3.10 in \cite{H16} or Remark 2 in \cite{Y22}). 
	\end{remark}
	There are versions of Harris's Theorem adapted to weaker Lyapunov conditions, providing subgeometric convergence results, see, e.g., \cite{BCG08, DFG09, DFMS04, DMT95}. 
	Here, we state and use a version which is found in \cite{H16, CM21, Y22}. We refer the reader, e.g., to \cite{DFG09} (Theorems 3.10 and 3.12) or to \cite{BCG08} (Theorem 1.2) for different versions of this theorem.

	\begin{thm}[Subgeometric Harris's Theorem] \label{thm:subgeo}
		Let $(S_t)_{t \geq 0}$ be a Markov semigroup satisfying the following hypotheses:
		\begin{hypothesis} [Weaker Foster-Lyapunov condition]\label{hyp:FL2} 
			There exist constants $\zeta> 0, \, D \geq 0$ and a continuous function $\phi : \Omega \to [1, +\infty)$ with pre-compact sub-level sets such that 
			\begin{align}
				\label{FL2} \tag{FL2}
				\mc L^* [\phi] (z) \leq D -\zeta h (\phi),
			\end{align} where $h: \R_+ \to \R$ is a strictly concave, positive, increasing function and $\lim_{u \to +\infty} h'(u)=0$.
		\end{hypothesis}
		
		\begin{hypothesis}[Minorisation condition]\label{hyp:M2} For every $R>0$, there exist a probability measure $\nu$, a constant $\beta \in (0,1)$ and some time $\tau >0 $ such that
			\begin{align}
				\label{M2} \tag{M2}
				S_\tau \delta_z \geq \beta \nu \quad \mbox{for all }z \in \mc C,
			\end{align}  where $\mathcal{C} =  \{ z \colon \phi(z) \leq R \}$.
		\end{hypothesis}
		Then $(S_t)_{t \geq 0}$ has a unique invariant measure $\mu_\infty$ satisfying
		\begin{align*}
			\int h(\phi(z)) \mu_\infty (\d z)\leq D,
		\end{align*} and there exists a constant $C$ such that 
		\begin{align*}
			\|S_t \mu - \mu_\infty\|_{\mathrm {TV}}  \leq \frac{C \mu (\phi)}{H_h^{-1}(t)} + \frac{C}{(h \circ H_h^{-1})(t)}
		\end{align*} holds for every $\mu(\phi) = \int \phi(z) \mu (\d z)$ where the function $H_h$ is defined by
		\begin{align*}
			H_h := \int_1^u \frac{\d s}{h(s)}. 
		\end{align*}
	\end{thm}
	The proofs of these theorems can be found in \cite{CM21,DFG09,H16}. In \cite{DFG09,H16}, the authors make a weaker assumption, namely $h(u) \leq u$ for any $u \geq 1$, instead of the one which is stated above $\lim_{u \to +\infty} h'(u)=0$. Their assumption allows for linear growth at infinity, whereas $\lim_{u \to +\infty} h'(u)=0$ essentially requires $h$ to be flat at infinity.
	
	For the case of angularly dependent jump kernels, we are able to recover exponential convergence using Harris's theorem (Theorem \ref{thm:harris}). The Foster-Lyapunov condition is proven in a similar way though more intricate than the work in our previous paper \cite{EY21}. The minorisation condition is considerably more challenging due to the fact that we need to track the dynamics over many more jumps to produce a lower bound. Therefore, we present the proof that the minorsiation condition holds in dimension $d=2$  to simplify the computations of the lower bound. Thus, the convergence result is stated in $d=2$ in Theorem \ref{thm:harris}. However, we believe that there is no serious obstacle to generalising it to higher dimensions.

	For the case of unbounded velocity spaces, we are only able to show subgeometric rates of convergence. We do this via the subgeometric version of Harris's theorem (Theorem \ref{thm:subgeo}). Here, we are able to use exactly the same uniform minorisation condition as in our previous paper \cite{EY21} in the linear case. The Foster-Lyapunov condition is substantially different and this is reflected in the subgeometric rate of convergence. 
	
	Before verifying these hypotheses for Eq. \eqref{eq:rt} we would like to add some comments on the Foster-Lyapunov condition. In order to show that \eqref{FL1} holds  true for Eq. \eqref{eq:rt} we would like to find some function $\phi(z)$ where $\phi(z) \rightarrow \infty$ as $|z| \rightarrow \infty$ and the existence of some $\tau>0$, $C>0$ and $\alpha \in (0,1)$ such that 
	\begin{align} \label{ineq:lyapunov1}
		\int_{\R^d} \int_{\mc V}\phi(x,v) f(\tau, x,v) \d x \d v \leq \alpha \int_{\R^d} \int_{\mc V} \phi(x,v )  f_0(x,v) \d x \d v + C, 
	\end{align} for any initial data $f_0(x,v) \in \mathcal{P}(\R^d\times \mathcal{V})$. This is because for $f$ satisfying $\partial_t f = \mathcal{L}[ f]$, \eqref{ineq:lyapunov1} is equivalent to showing that
	\begin{align} \label{ineq:lyapunov2}
		\mathcal{L}^* [\phi] \leq D - \zeta \phi , 
	\end{align} where $\mathcal{L}^*$ is the formal adjoint of $\mathcal{L}$ and $\zeta =  \frac{\log \alpha}{\tau}$ and $D = C \frac{\log \alpha}{\tau (1+\alpha)}$.
	In our case we have 
	\begin{align}
		\label{eq:L_generator}
		\mathcal{L}[f] (x,v) = - v\cdot \nabla_x f + \int_{\mathcal{V}} \lambda(v'\cdot \nabla_x M(x))\kappa(v',v) f(t, x,v')\d v' - \lambda(v \cdot \nabla_x M(x)) f(t, x,v), 
	\end{align}and therefore, 
	\begin{align}
		\label{eq:L_adjoint}
		\mathcal{L}^*[\phi] (x,v)= v\cdot \nabla_x \phi (x,v)+ \lambda(v\cdot \nabla_x M(x)) \left( \int_{\mathcal{V}}\kappa(v,v')\phi(x,v') \d v' - \phi(x,v)\right). 
	\end{align}

	\section{Angle-dependent tumbling kernel} \label{sec:angle-dependent}
	
	This section is dedicated to the long-time behaviour of the linear run and tumble equation, Eq. \eqref{eq:rt} with the angularly dependent tumbling kernel \eqref{asm:K_1}. 
	The following two sections are dedicated to verifying the two hypotheses of Harris's theorem (Theorem \ref{thm:harris}). At the end of this section, we provide a proof of Theorem \ref{thm:angulardep}.
	
	\subsection{Minorisation condition}
	
	In this section, we prove that Hypothesis \ref{hyp:M1}, \eqref{M1}, is satisfied for Eq. \eqref{eq:rt} with the angular-dependent tumbling kernel \eqref{asm:K_1}. Our overall strategy is as follows: we show that the solution $f \equiv f_t$ fo Eq. \eqref{eq:rt} attains a lower bound using a Duhamel's formula and bounding below by the part of $f_t$ where there are a particular number of jumps. 
	
	To keep the exposition of our strategy clear, we provide the following computations for only $d=2$. The very same approach can be used to extend the results to higher dimensions. 
	
	Denoting  $v_\theta := (\cos\theta, \sin \theta)$ and $\nabla_{x,y} g:= (\partial_x g, \partial_y g)$ for any differentiable function $g$, Eq. \eqref{eq:rt} becomes 
	\begin{align} \label{eq:rt_2d}
		\partial_t f + v_\theta  \cdot \nabla_{x,y} f = \int b(\theta'-\theta) \lambda( v_\theta \cdot \nabla_{x,y} M) f(t, x,y, \theta') \mathrm{d}\theta' - \lambda( v_\theta \cdot \nabla_{x, y} M)f  .
	\end{align}
	Next, we define the semigroup $(T_t)_{t\geq 0}$ accounting for the transport part,
	\begin{align*}
		(T_t f)(t, x,y, \theta) := f( t, x -t \cos \theta, y - t \sin \theta, \theta), 
	\end{align*}
	and the operator $\tilde L$, 
	\begin{align*}
		(\tilde{L}f)(t, x,y,\theta) := \int \1_{|\theta'-\theta|< \alpha} (\theta') f(t, x, y, \theta') \d\theta' = \int \1_{|\theta'|< \alpha} (\theta') f(t, x,y, \theta + \theta')\d \theta'.
	\end{align*}
	Then we have the following lemma
	\begin{lem} \label{lem:duhamel} 
		Suppose that $b(\theta) \geq \beta \1_{|\theta| \leq \alpha} (\theta )$. Then for any $n \geq 1$ we have 
		\begin{multline*} f(t,x, y,\theta ) \geq \\
			\beta ^n(1-\chi)^n e^{-(1+\chi)t}\int_0^t \int_0^{t_n} \dots \int_0^{t_2} \int_0^{t_1} \left(T_{t-t_n} \tilde{L} T_{t_n - t_{n-1}} \dots  \tilde{L} T_{t_2-t_1} \tilde{L} T_{t_1} f_0\right)(x, y, \theta) \mathrm{d}t_1 \d t_2 \dots \d t_n. 
		\end{multline*}
	\end{lem}
	\begin{proof}
		We have that
		\begin{multline*} 
			\p_t( f(t, x- t\cos \theta, y - t \sin \theta, \theta) =\lambda(v_\theta \cdot \nabla_{x,y} M(x-t\cos \theta, y - t\cos \theta)) \times\\
			\left( \int b(\theta') f(t, x-t \cos \theta, y - t \sin \theta, \theta + \theta') \d  \theta' - f(t, x-t\cos \theta, y-t \sin \theta, \theta ) \right).  
		\end{multline*} 
		Therefore writing $\Lambda (t,x,y,\theta) = \int_0^t \lambda( v_\theta \cdot \nabla_x M(x-s\cos \theta, y - s\cos \theta))\d  s$ we obtain
		\begin{multline*}
			f(t, x-t\cos \theta, y - t\sin \theta, \theta) = e^{-\Lambda(t)}f_0(x,y,\theta) \\ 
			+ e^{-\Lambda(t)}\int_0^t\lambda(v_\theta \cdot \nabla_x M(x-s\cos \theta, y - s\cos \theta)) \int b(\theta') f(t, x-s \cos \theta, y - s \sin \theta, \theta + \theta') \d \theta \d s. 
		\end{multline*} 
		Changing variables and using the fact that $1-\chi \leq \lambda \leq 1+\chi$, we obtain
		\begin{align*}
			f(t,x,y, \theta) \geq e^{-(1+\chi)t}T_t f_0 + (1-\chi)\int_0^t e^{-(1+\chi)(t-s)} \int b(\theta') (T_{t-s}f_s)(x, y, \theta + \theta') \d \theta'  \d s.
		\end{align*}
		Then using the assumption on $b$ we have
		\begin{align*}
			f(t,x,y, \theta) \geq e^{-(1+\chi)t}T_t f_0 + \beta (1-\chi)\int_0^t e^{-(1+\chi)(t-s)} T_{t-s} \tilde{L} f_s \d s. 
		\end{align*}
		We note that the second term above is positive, so we have
		\begin{align*}
			f(t,x,y, \theta) \geq e^{-(1+\chi)t}T_t f_0. 
		\end{align*}
		Then, we can plug this into the integral term to obtain
		\begin{align*}
			f(t,x,y, \theta)  \geq \beta (1-\chi) e^{-(1+\chi)t}\int_0^t T_{t-s} \tilde{L} T_s f_0 \d s. 
		\end{align*}
		We then continue substituting this into the integral term $n$ times to obtain the result for any $n \geq 1$. This finishes the proof.
	\end{proof}
	
	Next, we want to show that we can find an $n \geq 1$ and a range of admissible times for which we can provide a lower bound on the term $T_{t-t_n} \tilde{L} T_{t_n - t_{n-1}} \dots \tilde{L} T_{t_3-t_2} \tilde{L} T_{t_2-t_1} \tilde{L} T_{t_1} f_0$. 
	
	Let us define $P_t g := \tilde L T_t {g}$ for any probability measure $g$. Then we write $P_{t_1,t_2,t_3}^3 := P_{t_1} P_{t_2} P_{t_3}$ and note that for any measure $g$, $P_{t_1,t_2,t_3}^3 g$ has a density. In particular, $P_{t_1,t_2,t_3}^3 \delta_{x_0} \delta_{y_0} \delta_{\theta_0} (x,y, \theta)$ is bounded below by a uniform measure of a set with non-empty interior. Hence, we can bound it below by a constant times the indicator function of red a ball in the position variable $x$ with a centre at a point which we can compute. 
	We can then use this computation to show that $P^3_{t_1, t_2, t_3}$ acting on the indicator function of a ball is bounded below by a constant times the indicator function of another ball with a slightly larger radius and a different centre. By tracking how these centres move and how the radii of the balls grow, we can then get a bound below in $(x,y)$-space by the indicator function of a ball whose centre is at the origin (rather than depending on the initial point). 
	
	We then need to prove that we can subsequently reach all possible angles (i.e., velocities) while maintaining a uniform lower bound below in position. We show this by looking at repeated jumps in a small time period so that we reach all angles without moving too far. This then allows us to reach all possible angles but slightly shrinks the ball we found in the lower bound for the spatial variables.

	Therefore, in the following lemma, we look at $P_{t_1,t_2,t_3}^3 \delta_{x_0} \delta_{y_0} \delta_{\theta_0} (x,y, \theta)$ and show that it is bounded below by a constant times the indicator function of a set which contains a ball whose radius and centre we can compute. 
	
	\begin{lem} \label{lem:crecent}
		Let $\alpha < \dfrac{\pi}{2}$ and $r_1, r_2, r_3 >0 $ are a set of times, $\varepsilon >0$ and $|s_i-r_i| < \varepsilon$ for each $i \in \{1,2,3\}$. Then if $\varepsilon$ is sufficiently small we have 
		\begin{align*}
			P^3_{s_1,s_2,s_3} \delta_{x_0} \delta_{y_0} \delta_{\theta_0}(x, y, \theta) \geq \frac{1}{s_2 s_3} \1_{B((x_*,y_*); r)} (x,y)\1_{|\theta-\theta_0 \pm \delta \theta| \leq \frac{\alpha}{2}} (\theta),
		\end{align*} where  
		\begin{align*}
			x_* = x_0 &+ (r_1+R)\cos(\theta_0),  \quad y_* = y_0 + (r_1+R) \sin(\theta_0), \quad  r= r_2r_3(1-\cos(\alpha/2)), \\
			&\delta\theta =  \pm \arctan \left( \frac{r_2 \sin(\alpha/2)}{r_3+r_2\cos(\alpha/2)} \right),  \quad R= \sqrt{ r_2^2 + r_3^2 + 2r_2r_3 \cos(\alpha/2)}.
		\end{align*} 
	\end{lem}
	\begin{proof}
		Since we have 
		\begin{multline*}
			T_{s_1}(\delta_{x_0}\delta_{y_0} \delta_{\theta_0}) (x,y,\theta) = \delta_{x_0 + s_1 \cos(\theta_0)}(x)\delta_{y_0 + s_1 \sin(\theta_0)}(y) \delta_{\theta_0} (\theta) 
			\\= \delta_{x_0}(x-s_1\cos (\theta_0)) \delta_{y_0} (y - s_1 \sin (\theta_0)) \delta_{\theta_0} (\theta),
		\end{multline*} and
		\begin{align*}
			\tilde{L}T_{s_1}(\delta_{x_0}\delta_{y_0} \delta_{\theta_0}) (x,y,\theta)=\int_{-\alpha}^\alpha \delta_{x_0 + s_1 \cos(\theta_0)}(x) \delta_{y_0 + s_1 \sin(\theta_0)}(y) \delta_{\theta_0}(\theta-\theta') \mathrm{d}\theta',
		\end{align*}
		applying $T_{s_2}$ once more we obtain,
		\begin{multline*}
			T_{s_2} \tilde{L} T_{s_1}(\delta_{x_0}\delta_{y_0} \delta_{\theta_0}) (x,y, \theta) \\= \int_{-\alpha}^\alpha \delta_{x_0}( x-s_1 \cos(\theta_0) - s_2 \cos(\theta_0 + \theta_1)) \delta_{y_0}(y-s_1 \sin(\theta_0) - s_2 \sin(\theta_0 + \theta_1)) \delta_{\theta_0}(\theta-\theta_1) \mathrm{d} \theta_1.  
		\end{multline*}
		Iterating this yields
		\begin{multline*}
			\tilde{L}T_{s_3}\tilde{L} T_{s_2} \tilde{L} T_{s_1} (\delta_{x_0}\delta_{y_0} \delta_{\theta_0}) (x,y, \theta)  \\ = \int_{-\alpha}^\alpha \int_{-\alpha}^\alpha \int_{-\alpha}^\alpha \delta_{x_0}(x-s_1 \cos(\theta_0) - s_2 \cos(\theta_0 + \theta_1) - s_3 \cos(\theta_0 + \theta_1 + \theta_2)) \\
			\times \delta_{y_0}(y-s_1 \sin(\theta_0) - s_2 \sin(\theta_0 + \theta_1) - s_3 \sin (\theta_0 + \theta_1 + \theta_2))\delta_{\theta_0}(\theta - \theta_1 - \theta_2 - \theta_3) \d \theta_1 \d \theta_2 \d \theta_3.
		\end{multline*}
		We perform a change of variables, 
		\begin{align*}
			\tilde{x} = s_2 \cos(\theta_0+\theta_1) + s_3 \cos(\theta_0 + \theta_1+\theta_2), \quad   \tilde{y} = s_2 \sin(\theta_0 + \theta_1) + s_3 \sin(\theta_0+\theta_1+\theta_2),
		\end{align*} then we have 
		\begin{align*}
			\frac{\partial \tilde{x}}{\partial \theta_1} = -s_2 \sin(\theta_0 + \theta_1) -s_3\sin(\theta_0 + \theta_1 +\theta_2), \quad \frac{\partial \tilde{y}}{\partial \theta_1} = s_2 \cos(\theta_0 + \theta_1) + s_3 \cos(\theta_0 + \theta_1 + \theta_2)
		\end{align*} and
		\begin{align*}
			\frac{\partial\tilde{x}}{\partial \theta_2} = -s_3 \sin(\theta_0 + \theta_1 + \theta_2), \quad \frac{\partial \tilde{y}}{\partial \theta_2} = s_3 \cos(\theta_0 + \theta_1 + \theta_2). 
		\end{align*}
		The Jacobian of this change of variables is $\d \tilde{x} \d \tilde{y} = s_2 s_3 \sin(\theta_2) \d \theta_1 \d \theta_2$. 
		
		Moreover, defining $\tilde{\theta} := \theta_1 + \theta_2 +\theta_3$ and then $\mathrm{d}\tilde{\theta} = \mathrm{d}\theta$, we obtain 
		\begin{multline*}
			\tilde{L}T_{s_3}\tilde{L} T_{s_2} \tilde{L} T_{s_1} (\delta_{x_0}\delta_{y_0} \delta_{\theta_0}) (x,y,\theta)= \int_{S(x_0,y_0,\theta_0, s_1,s_2,s_3)} \delta_{x_0}(x-s_1 \cos(\theta_0) - \tilde{x})) \\
			\times \delta_{y_0}(y-s_1 \sin(\theta_0) - \tilde{y})\delta_{\theta_0}(\theta - \tilde{\theta})\frac{1}{s_2s_3 |\sin(\theta_2(\tilde{x},\tilde{y}))|} \d \tilde{x} \d \tilde{y} \d \tilde{\theta},
		\end{multline*}
		where $S(x_0,y_0,\theta_0,s_1,s_2,s_3)$ is the set of possible values of $x_0 + s_1 \cos(\theta_0) +\tilde{x}, y_0 + s_1 \sin(\theta_0) +\tilde{y}$, and $\theta_0 + \tilde{\theta}$, (see \textbf{\em Figure 1} below). 
		Then, as $|\sin(\theta_2) | \leq 1$ we have 
		\begin{align*}
			\tilde{L}T_{s_3}\tilde{L} T_{s_2} \tilde{L} T_{s_1} (\delta_{x_0}\delta_{y_0} \delta_{\theta_0}) (x,y,\theta) \geq \frac{1}{s_2 s_3} \mathds{1}_{S(x_0,y_0,\theta_0, s_1, s_2, s_3)} (x,y, \theta).
		\end{align*}
		Next, we want to show that $S(x_0,y_0,\theta_0, s_1, s_2, s_3)$ contains the ball mentioned in the statement of the lemma. We notice that
		\begin{align*}
			\tilde{x}^2+ \tilde{y}^2 = s_1^2 +s_2^2 + 2 s_1 s_2 \cos(\theta_2) = R(\theta_2)^2. 
		\end{align*} 
		We choose $\theta_2 = \pm \frac{\alpha}{2}$ and 
		$\theta_1  =  \pm \arctan \left( \frac{s_3 \sin(\alpha/2)}{s_2 +s_3\cos(\alpha/2)} \right)$ and this will give the centre of the ball.
		
		Now for a given $\theta_2$,  we choose $\beta =  \arctan \left( \frac{s_3 \sin(\theta_2)}{s_2 +s_3\cos(\theta_2)} \right)$ then we can write
		\begin{align*}
			\tilde{x} &= (s_2 \cos (\beta )+ s_3 \cos(\theta_2 - \beta)) \cos(\theta_0 + \theta_1 + \beta) = R(\theta_2) \cos(\theta_0 + \theta_1 + \beta) \\
			\tilde{y} &= (s_2 \cos (\beta)+ s_3 \sin(\theta_2 - \beta)) \sin(\theta_0 + \theta_1 + \beta) = R(\theta_2) \sin(\theta_0 + \theta_1 + \beta).
		\end{align*}
		So we can set $\tilde{x} = R\cos(\omega)$ and $\tilde{y}=R \sin(\omega)$ as long as we can choose 
		\begin{align*}
			\theta_2 = \arccos \left  ( \frac{R^2-s_2^2-s_3^2}{s_2s_3} \right ) \, \mbox{   and   } \,  \theta_1 = \omega - \arctan \left ( \frac{s_3 \sin(\theta_2)}{s_2 + s_3 \cos(\theta_2)} \right ). 
		\end{align*} We can see that for $\theta_1$ and $\theta_2$ small we have
		\begin{align*}
			\theta_1 \approx \omega - \frac{s_3 \theta_2}{s_2 +s_3}.
		\end{align*}
		Thus, we expect to be able to choose $|\beta|$ up to $\frac{s_3 \alpha}{s_2+s_3}$ which is at least as large as it is needed to cover the ball given that $\alpha$ is small. Then since all the computations depend continuously on $s_1, s_2, s_3$ and the ball is strictly contained inside the set $S$, we can replace $s_1, s_2, s_3$ with $r_1, r_2, r_3$ provided that $\varepsilon$ is small enough.  This concludes the proof. 
	\end{proof}
	\begin{center}
		\fbox{ 
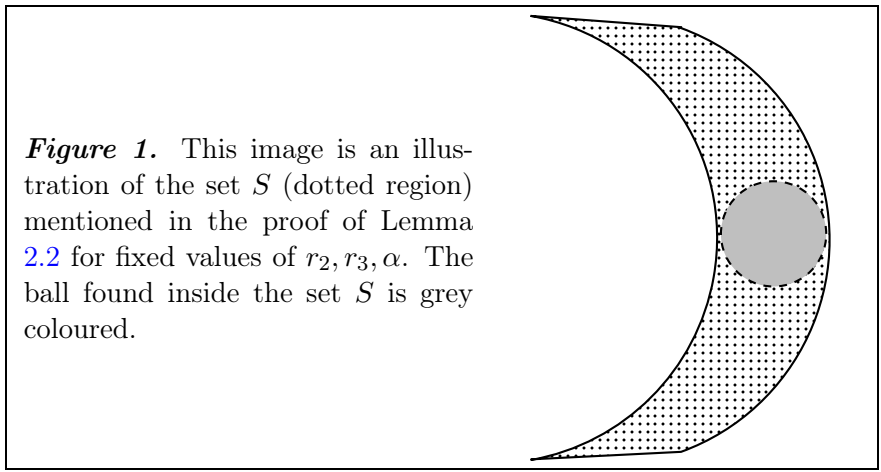
			\begin{minipage}[c]{0.35\textwidth}
				\textbf{\em Figure 1.} \label{fig1}
				This image is an illustration of the set $S$ (dotted region) mentioned in the proof of Lemma \ref{lem:crecent} for fixed values of $r_2, r_3, \alpha$. The ball found inside the set $S$ is grey coloured.
			\end{minipage}%
			\begin{minipage}[c]{0.3\textwidth}
				\raggedleft
				\begin{tikzpicture} 
					\draw[thick, pattern=dots] (0,0) -- (2,0.1)  arc(-70:70:3)  -- (2,5.75) -- (0,5.9) arc(80:-80:3);
					\draw[thick, fill=lightgray, dashed] (3.23,3) circle (0.7cm);
				\end{tikzpicture}
			\end{minipage}
		}
	\end{center}
	
	\,
	
	Using Lemma \ref{lem:crecent} we then prove the following:
	
	\begin{lem} \label{lem:crecent_ball} Let $r_1, r_2, r_3 >0$ be a set of times, $R, \alpha$ are as in Lemma \ref{lem:crecent}. Then we have 
		\begin{align*}
			P^3_{r_1, r_2, r_3} \mathds{1}_{B((x_0,y_0);\tilde{r})}(x,y) \1_{|\theta-\theta_0|<\frac{\alpha}{2}} (\theta) \geq \gamma \1_{B(x_{**}, y_{**}; \tilde{r} + \frac{r}{2})} (x,y)\1_{|\theta_{**}-\theta| \leq \alpha} (\theta),
		\end{align*}
		where 
		\begin{align*}
			x_{**} = x_0 - (r_1+R) \cos(\theta_0+\delta\theta), \quad y_{**} = y_0 - (r_1+R) \sin(\theta_0+\delta\theta), \quad \theta_{**}= \theta_0 + \delta\theta
		\end{align*} and $\gamma$ a constant that can be computed explicitly. Let us define the map $F$,
		\begin{align} \label{eq:F}
			F(x_0,y_0,\theta_0) := (x_{**}, y_{**}, \theta_{**})
		\end{align} which tracks how the centres of the balls move. 
	\end{lem}
	\begin{proof}
		We have
		\begin{align*}
			\1_{B((x_0,y_0);\tilde{r})}(x,y) \1_{|\theta-\theta_0|<\frac{\alpha}{2}} (\theta)= \int \delta_{x'}(x) \delta_{y'}(y) \delta_{\theta'}(\theta)  \1_{B((x_0,y_0);\tilde{r})}(x',y')\1_{|\theta'-\theta_0|< \frac{\alpha}{2}} (\theta)\d x' \d y' \d \theta ',
		\end{align*} where $\int $ represents the triple integral, $\int_{\R^2 \times \R^2} \d x' \d y' \int_{-\alpha}^{\alpha} \d \theta'$. 
		Now applying three times $T_t$ and $\tilde L$ yields,
		\begin{align*}
			&P^3_{r_1, r_2, r_3} \1_{B((x_0,y_0);\tilde{r})}(x,y) \1_{|\theta-\theta_0|<\frac{\alpha}{2}} (\theta)\\
			\geq & \int \left(P^3_{r_1, r_2, r_3} \delta_{x'} \delta_{y'}\delta_{\theta'} \right) (x,y,\theta) \1_{B((x_0,y_0);\tilde{r})}(x',y')\1_{|\theta'-\theta_0|<\frac{\alpha}{2}} \d x' \d y' \d \theta ' \\
			\geq	& \frac{1}{r_2 r_3}\int \1_{B((x' +(r_1+R) \cos(\theta'), y'+(r_1+R) \sin(\theta')); r)}(x,y) \1_{B((x_0,y_0); \tilde{r})}(x',y') \1_{|\theta-\theta'-\delta\theta|< \frac{\alpha}{2}} \1_{|\theta'-\theta_0|< \frac{\alpha}{2}} \d x' \d y' \d \theta' \\
			= & \frac{1}{r_2 r_3}\int \1_{B((x -(r_1+R) \cos(\theta'), y-(r_1+R) \sin(\theta')); r)}(x',y') \1_{B((x_0,y_0); \tilde{r})}(x',y') \1_{|\theta-\theta'-\delta\theta|<\frac{\alpha}{2}}  \1_{|\theta'-\theta_0|< \frac{\alpha}{2}}  \d x' \d y' \d \theta' \\
			= & \frac{1}{r_2 r_3} \int  \left | B \left ((x -(r_1+R) \cos(\theta'), y-(r_1+R) \sin(\theta')); r \right ) \cap B \left ((x_0,y_0); \tilde{r} \right ) \right |  \1_{|\theta-\theta'-\delta\theta|<\frac{\alpha}{2}} \1_{|\theta'-\theta_0|< \frac{\alpha}{2}} \d  \theta'
		\end{align*} In the last line above, we fixed $\theta'$ and considered the integral in $x'$ and $y'$ which will measure the size of the overlap between the balls $B((x-(r_1+R) \cos \theta', y-(r_1+R) \sin \theta') ; r)$ and $B((x_0,y_0); \tilde{r})$. 
		
		If $(x, y) \in B((x_0+(r_1+R)\cos \theta', y_0 +(r_1+R) \sin \theta'); \tilde{r} + \frac{r}{2})$ then we can bound the size of the overlap below by $\frac{\pi r^2}{4}$. We recall from Lemma \ref{lem:crecent} that $r = r_2 r_3 \left (1- \cos  (\frac{\alpha}{2})\right )$, so we have
		\begin{multline*} \left | B \left ((x -(r_1+R) \cos(\theta'), y-(r_1+R) \sin(\theta')); r \right ) \cap B \left ((x_0,y_0); \tilde{r} \right ) \right |  \geq \\ \frac{\pi r_2 r_3}{4} \left (1- \cos  \left  ( \frac{\alpha}{2} \right) \right )^2\1_{B((x_0+(r_1+R)\cos \theta', y_0 +(r_1+R) \sin \theta'); \tilde{r} + \frac{r}{2})}(x,y). 
		\end{multline*}
		Therefore the lower bound becomes, denoting $C(r_2,r_3, \alpha) : = \frac{\pi r_2 r_3}{4} \left (1- \cos  \left  ( \frac{\alpha}{2} \right) \right )^2$ (a constant depending on $r_2,r_3, \alpha$),
		\begin{align*}
			P^3_{r_1, r_2, r_3} &\1_{B((x_0,y_0);\tilde{r})}(x,y) \1_{|\theta-\theta_0|<\frac{\alpha}{2}} (\theta)\\
			&\geq   C(r_2,r_3, \alpha)  \int \1_{B \left ((x_0+(r_1+R)\cos \theta', y_0 +(r_1+R) \sin \theta'); \tilde{r} + \frac{r}{2} \right )}(x,y)\1_{|\theta-\theta'-\delta\theta|<\frac{\alpha}{2}} \1_{|\theta'-\theta_0|< \alpha}  \d \theta', \\
			&= C(r_2,r_3, \alpha)  \int \1_{B \left ((x_0  - x, y_0 -y); \tilde{r} +\frac{r}{2} \right )}((r_1+R)\cos(\theta'), (r_1+R) \sin(\theta')) \1_{|\theta-\theta'-\delta\theta|<\frac{\alpha}{2}} \1_{|\theta'-\theta_0|< \alpha}  \d \theta' \\
			&= C(r_2,r_3, \alpha)  \int \1_{ \left |\theta' - \arccos \left ( \frac{y_0-y}{x_0-x}\right ) \right | \leq \arccos \left ( \frac{\tilde{r}+r/2}{r_1 +R} \right )}   \1_{|\theta-\theta'-\delta\theta|<\frac{\alpha}{2}}  \1_{|\theta'-\theta_0|< \alpha} \d \theta' 
		\end{align*}
		Thus,  if we impose that $(x_0-x, y_0 - y) \in B \left ( \left ( (R+r_1)\cos(\theta_0 + \delta\theta), (R+r_1) \sin(\theta_0 + \delta\theta) \right ); \tilde{r} + \frac{r}{4} \right )$ then we can bound the last integral above by 
		\begin{align*}
			\gamma \1_{B \left ((x_0 - (r_1+R)\cos(\theta_0 + \delta \theta), y_0 -(r_1+R) \sin(\theta_0 + \delta\theta));\tilde{r} + \frac{r}{4} \right )} (x,y) \1_{|\theta - \theta_0 - \delta\theta|< \frac{\alpha}{2}}(\theta), 
		\end{align*} where $\gamma$ is a constant depending on $r_1, r_2, r_3$ and is bounded uniformly in terms of upper and lower bounds on $r_1, r_2, r_3$. Roughly we can compute this constant $ \gamma \approx \frac{\pi r_2^2 r_3^2 \alpha^2}{32(r_1+R)}$.
	\end{proof}

	As we iterate the process that moves the centre of the balls, we can find a radius $\hat R >0$ so that the path of centres stays inside $B((x_0,y_0); \hat R)$ forever. We prove this in the following lemma.

	\begin{lem} \label{lem:big_ball}
		Let $k \geq 1$ and consider a sequence of maps $(x_k, y_k, \theta_k) = F^{k}(x_0, y_0, \theta_0)$ where $F$ is defined in Lemma \ref{lem:crecent_ball}. Then there exists a constant $\hat{R}>0$ depending on $r_1, r_2, r_3$ such that $(x_k, y_k) \in B((x_0, y_0);\hat R)$ for every $k \geq 1$.
	\end{lem}
	\begin{proof}
		We have 
		\begin{align*}
			x_k = x_0 - (r_1 + R) \cos(\theta_0 + \delta\theta) - (r_1+R) \cos(\theta_0 + 2 \delta\theta) - \dots - (r_1+R)\cos(\theta_0 + k\delta \theta),
		\end{align*}	
		and $y_k$ is defined similarly. The points $(x_k, y_k)$ are illustrated in \textbf{\em Figure 2} below.
		
		We can then compute the ball that all the points lie on, which is a circle whose radius $\hat R$ and centre $C$ are given by
		\begin{align*}
			\hat R = \frac{r_1+R}{\sin(\delta \theta/2)}, \quad C = (x_0,y_0) - \hat{R} \left ( \cos( \pi/2 - (\theta_0 +\delta \theta/2)), \sin( \pi/2 - (\theta_0 +\delta \theta/2)) \right ).
		\end{align*} This finishes the proof.
	\end{proof}
	
	\fbox{
		\begin{minipage}[c]{0.45\textwidth}
			
			\begin{tikzpicture} \label{fig2}
				\raggedleft
				\draw [thick] (-5,0)-- (0,0) node[midway, below] {$\hat R$};
				\filldraw[black] (0,0) circle (2pt) ;
				\draw (-1,0) arc (180:150:1);
				\draw (-0.85,0.5) arc (150:125:1);
				\draw [thick] (0,0) -- (-4.5,2.5)  node[anchor=east] {$(x_1,y_1)$};
				\filldraw[black] (-4.5,2.5) circle (2pt) ;
				\draw[thick] (-4,2.9) arc (30:60:1)  node[midway, anchor= south west  ] {$\delta \theta$};
				\draw [thick] (0,0) node[anchor=west] {$C$} -- (-5,0) node[anchor= east] {$(x_0,y_0)$} ;
				\filldraw[black] (-5,0) circle (2pt) ;
				\draw [thick] (-5,0) -- (-4.5,2.5) -- (-2.5, 4) 	node[anchor= south east] {$(x_2,y_2)$} ;
				\filldraw[black] (-2.5,4) circle (2pt) ;
				\draw[thick]  (-2.5, 4) -- (0,0);
				\draw [thick] (-2.5,4) -- (-0.8,4.4);
				\draw[thick, dashed] (-4.5,2.5) -- (-4.15, 4) ;
				\draw[thick] (-1.7,4.2) arc (10:30:1)  node[anchor = west] {$\delta \theta$};
				\draw[thick, dashed] (-2.5, 4)-- (-1.5, 4.75);
				\draw[thick, dashed] (-5, 0)-- (-5.5, 1.4);
				\draw (-4.85,0.7) arc (85:108:1) node [midway,anchor= south ] {$\delta \theta$};
			\end{tikzpicture}
		\end{minipage}%
		\begin{minipage}[c]{0.45\textwidth}
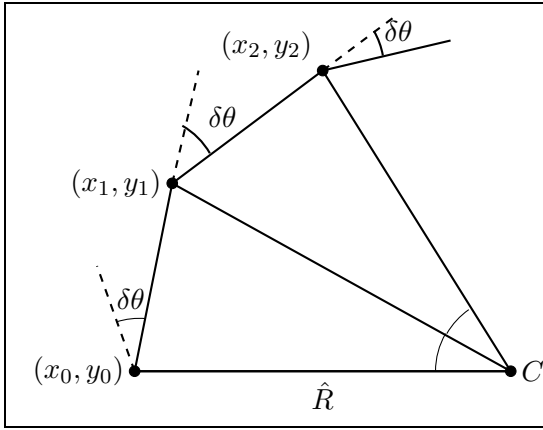
			\textbf{\em Figure 2.} This image shows the iterates of the map $F$ in $(x,y)$-space. We notice that each triangle formed by the points $C, (x_k, y_k), (x_{k+1}, y_{k+1})$ is an isosceles triangle and a rotation of the triangle formed by the points $C, (x_0, y_0), (x_1, y_1)$. We can then compute the circle that they all lie on and the point at the centre $C$ using standard tools in trigonometry.
		\end{minipage}
	}
	
	\medskip
	
	In the next lemma, we prove that after $\frac{24\hat R}{r}$ steps, the $x,y$ marginal is bounded by a constant times the indicator function of a ball of radius $\hat R$.
	
	\begin{lem} \label{lem:spacebound}
		If we let $\tilde n  = \big \lceil \frac{24\hat R}{r} \big  \rceil$, then there exists some constant $\tilde \gamma  >0$ such that
		\begin{align} \label{eq:spacebound}
			P^{\tilde n } \delta_{x_0} \delta_{y_0} \delta_{\theta_0} (x,y, \theta)\geq \tilde  \gamma \1_{B((x_0, y_0); \hat R)} (x,y)\ \1_{|\theta-\theta_k| \leq \alpha} (\theta).
		\end{align}
	\end{lem}
	\begin{proof}
		By Lemma \ref{lem:big_ball} we can choose a path such that $(x_k,y_k) \in B((x_0,y_0); \hat R)$ and writing $k=\frac{\tilde n}{3}$ Lemma \ref{lem:crecent_ball} yields
		\begin{align*}
			P^{\tilde n } \delta_{x_0} \delta_{y_0} \delta_{\theta_0} (x,y, \theta )\geq \gamma^{k} \1_{B \left ((x_k, y_k); \left (1+\frac{k-1}{4} \right )r \right )} (x,y)\1_{|\theta-\theta_k| \leq \alpha}(\theta).
		\end{align*}
		Since we also have that $\left (1+\frac{k-1}{4} \right )r \geq 2 \hat R$, therefore we obtain \eqref{eq:spacebound} with $\tilde \gamma = \gamma^k$. 
	\end{proof}
	
	Next, we look at the angles and prove the following lemma:
	
	\begin{lem} \label{lem:thetabound}
		Let $n_*= \left \lceil \frac{4\pi}{\alpha}  \right \rceil$ and suppose that $t_1 +t_2 + \dots + t_{n_*} \leq l $. Then we have
		\begin{align*}
			P^{n_*}_{t_1, t_2, \dots, t_{n_*}} \big ( \1_{B\big ((0,0); \frac{\hat R}{2} \big )} (x,y)\1_{|\theta-\theta_0|< \alpha} (\theta)\big) \geq \gamma \1_{B\big  ((0,0); \frac{\hat R}{2} - 2 l\big )} , 
		\end{align*} uniform in $\theta$. 
	\end{lem}
	\begin{proof}
		First, let us look at 
		\begin{align*}
			\int P^{n_*} \delta_{x_0}\delta_{y_0} &\delta_{\theta_0} (x,y,\theta)\d x \d y \\ &= \int \int P^{n_* -1}(\delta_{x_0}\delta_{y_0} \delta_{\theta_0})(x-t_{n_*}\cos(\theta),y - t_{n_*} \sin \theta, \theta+\theta') \1_{|\theta'| \leq \alpha} (\theta')\d x \d y \d \theta' \\
			&= \int \int \int P^{n_* -1}(\delta_{x_0}\delta_{y_0} \delta_{\theta_0}) (x,y, \theta+\theta') \1_{|\theta'| \leq \alpha} (\theta') \d x \d y \d \theta' \mathrm{d}\theta. 
		\end{align*} 
		We can keep iterating this to obtain
		\begin{align*} 
			\int P^{n_*}  \delta_{x_0}\delta_{y_0} \delta_{\theta_0} &(x,y,\theta) \d x \d y 
			\\&= \int \int \cdots \int  \delta_{x_0}\delta_{y_0} \delta_{\theta_0}(\theta+\theta_1 +\theta_2 + \dots + \theta_{n_*}) \1_{|\theta_1| \leq \alpha} \dots \1_{|\theta_{n_*}| \leq \alpha} \d x \d y \d \theta_1 \dots \d\theta_{n_*} \\
			& = \int \delta_{\theta_0}(\theta+\theta_1 +\theta_2 + \dots + \theta_{n_*}) \1_{|\theta_1| \leq \alpha} \dots \1_{|\theta_{n_*}| \leq \alpha} \d \theta_1 \dots \d \theta_{n_*} \\
			& \geq C.
		\end{align*}
		For some constant $C$ that doesn't depend on $\theta_0, \theta$.
		Therefore,
		\begin{align*}
			\int P^{n_*} \delta_{x_0}\delta_{y_0} \delta_{\theta_0}(x,y,\theta) \d x \d y \geq C. 
		\end{align*} Hence we can write
		\begin{align*}
			P^{n_*} \delta_{x_0}\delta_{y_0} \delta_{\theta_0} (x,y, \theta) = \mu(\theta \mid  \theta_0)\mu( (x,y) \mid  \theta, (x_0, y_0, \theta_0))  \geq C \mu( (x,y) \mid \theta, (x_0, y_0, \theta_0)) 
		\end{align*}
		where $\mu( (x,y) \mid  \theta, (x_0, y_0, \theta_0))$ is the conditional law of the position variables given $\theta$ and the initial point $(x_0, y_0)$ so we have that
		\begin{align*}
			\int \mu( (x,y) \mid  \theta, (x_0, y_0, \theta_0)) \d x \d y = 1. 
		\end{align*}
		We can also see that 
		\begin{align*}
			\supp ( \mu( (x,y) \mid  \theta, (x_0, y_0, \theta_0)) \subseteq \1_{B((x_0,y_0);l)} 
		\end{align*}
		as the $x$ and $y$ variables can have travelled a distance of at most $l$. Then
		\begin{align*}
			P^{n_*} \big ( \1_{B((0,0); \hat R )} \1_{|\theta-\theta_*|< \alpha}\big ) \geq \int c \mu( (x,y) \mid \theta, (x', y', \theta')) \1_{B((0,0); \hat R)}(x',y') \1_{|\theta'-\theta_0| < \alpha} \d x' \d y' \d \theta'. 
		\end{align*}
		Then using the translation invariance of the transport map we can write
		\begin{align*}
			\mu((x,y )\mid \theta, (x',y', \theta'))= \mu((-x',-y' ) \mid  \theta, (-x,-y, \theta')),
		\end{align*} so we have
		\begin{align*}
			P^{n_*} \big ( \1_{B((0,0); \hat R)} 1_{|\theta-\theta_*|< \alpha} \big ) &\geq \int \mu((-x',-y') \mid  \theta, (-x, -y, \theta')) \1_{B((0,0); \hat R)}(x',y') \1_{|\theta'-\theta_0| < \alpha} \d x' \d y' \d \theta'  \\ &\geq  C \1_{B((0,0);\hat R-2l)}(x,y). 
		\end{align*}
		Above, we use the fact that if $(x, y) \in B((0,0);\hat R-2l)$ then we will be integrating over the full support of $\mu((-x',-y') \mid \theta, (-x, -y, \theta'))$.
	\end{proof}
	
	Using the previous lemmas, we can prove that the minorisation condition \eqref{M1} for Eq. \eqref{eq:rt} with the angularly dependent tumbling kernel is satisfied. 
	\begin{prp} \label{prop:minorisation_angle_dependent}
		Let $f_0 (x,y, \theta)= \delta_{x_0} \delta_{ y_0} \delta_{\theta_0} (x,y,\theta)$ with $(x_0, y_0) \in B \big ((0,0); \frac{\hat R}{2} \big )$ then, after $ \big  \lceil \frac{24 \hat R}{r} \big \rceil + n_*$ steps, provided that the times are chosen suitably, we have
		\begin{align*}
			f (t,x,y, \theta) \geq C \1_{B \big ((0,0); \frac{\hat R}{4} \big )},
		\end{align*} for some $C$ and $\hat R$ that can be computed explicitly.
	\end{prp}
	
	\begin{proof}
		Suppose that the first $n_{**}$ inter-jump times are within $\varepsilon$ of $r_1, r_2, r_3$ with $\varepsilon$ small enough so that Lemma \ref{lem:crecent} applies and $\gamma$ is chosen such that Lemma \ref{lem:spacebound} holds for any interjump times in this range with constant $\gamma$. Suppose further that the sum of the last $n_*$ jump times is less than $l$. Then, using the previous result we have
		\begin{align*}
			P^{n_{**} + n_*} \delta_{x_0}\delta_{y_0}\delta_{\theta_0} (x,y,\theta)&= P^{n_*} \left( P^{n_{**}} \delta_{x_0}\delta_{y_0} \delta_{\theta_0} \right) (x,y,\theta)
			\\ \mbox{using Lemma \ref{lem:spacebound}} \quad \quad \quad & \geq P^{n_*} \big (  \gamma^{\frac{n_{**}}{3}} \1_{B((x_0, y_0); \hat R)} \1_{\left |\theta - \theta_0 - n_{**}\frac{\delta \theta}{3}\right | \leq \frac{\alpha}{2}}\big )\\
			\mbox{using that $(x_0,y_0) \in  B \big ((0,0); \frac{\hat R}{2} \big )$} \qquad \quad & \geq P^{n_*} \big (  \gamma^{\frac{n_{**}}{3}} \1_{B((0, 0); \frac{\hat R}{2})} \1_{ \left |\theta - \theta_0 - n_{**}\frac{\delta \theta}{3} \right | \leq \frac{\alpha}{2}}\big )\\
			\mbox{using Lemma \ref{lem:thetabound}} \quad \quad \quad & \geq  \gamma \1_{B \left ((0,0); \frac{\hat R}{2}-2l \right )}.
		\end{align*}
		We then take $m = \frac{\hat R}{8}$ and substitute this into Lemma \ref{lem:duhamel} and integrate over the admissible possible jump times and obtain
		\begin{align*}
			f(t,x,v) \geq \hat \gamma \1_{B \big ((0, 0); \frac{\hat R}{4} \big )} 
		\end{align*}
		where $\hat \gamma$ is a constant we could, in principle, compute. This verifies \eqref{M1} for \eqref{eq:rt_2d} with the angularly dependent kernel \eqref{asm:K_1}. 
	\end{proof}

	\subsection{Foster-Lyapunov condition}
	\label{sec:Foster-Lyapunov1}
	In this section, we verify the Foster-Lyapunov condition \eqref{FL1} for Eq. \eqref{eq:rt} for the case \eqref{asm:K_1}. We remark that the minorisation condition in the previous section was given in dimension $d=2$ to keep the exposition clear. The Foster-Lyapunov condition we prove in this section is valid for arbitrary dimensions $d\geq 1$. 
	We start by proving the following lemma. 
	\begin{lem} \label{lem:FL1}
		Suppose that $\kappa$ is a collision kernel \eqref{asm:K_1} satisfying the hypothesis \textbf{\em (H4)} then 
		\begin{align} \label{C_kappa}
			\int_{\S^{d-1}} \kappa(v',v)v'\cdot \nabla_x M(x) \d v' = C_\kappa v \cdot \nabla_x M(x), 
		\end{align} where $C_\kappa \leq 1$ is a constant that only depends on the form of the collision kernel.
	\end{lem}
	\begin{proof} We perform a change of variables and write write $v' = \cos (\theta)  v + \sin (\theta) w$ where $w \in S_v$ ranges over the sphere of dimension $d-2$ lying in the hyperplane perpendicular to $v$. This gives us
		\begin{align*}
			\int_{\S^{d-1}} \kappa(v',v)v'\cdot \nabla_x M(x) \d v'  = 	\int_{-\pi}^{\pi} \int_{S_v} \kappa_1 (\theta) (\cos (\theta) v \cdot \nabla_x M (x)+ \sin (\theta) w \cdot \nabla M(x) ) \d w \d \theta. 
		\end{align*} Since $\kappa_1$ is an even function, integrating first in $\theta$ yields
		\begin{align*}
			|S_v| \Big ( \int \kappa_1 (\theta) \cos (\theta) \d \theta \Big) v \cdot \nabla_x M(x) = |\mathbb{S}^{d-2}| \Big( \int \kappa_1(\theta)\cos (\theta) \d \theta \Big) v \cdot \nabla_x M(x). 
		\end{align*}where $|S_v|$ is the size of $S_v$, similarly for $|\S^{d-2}|$. Therefore we obtain \eqref{C_kappa} with 
		\begin{align*}
			C_\kappa = |\mathbb{S}^{d-2}| \Big( \int \kappa_1(\theta) \cos (\theta)  \d \theta \Big). 
		\end{align*}
	\end{proof}
	
	\begin{lem} \label{lem:FL12}
		Suppose that $\kappa$ is a collision kernel \eqref{asm:K_1} and that hypotheses  \textbf{\em (H2)} and \textbf{\em (H4)} are satisfied. 
		Then we have
		\begin{align*}
			\int_{\S^{d-1}} \kappa(v',v) v' \cdot \nabla_x M(x) \psi(v' \cdot \nabla_x M (x) ) \d v' \geq \tilde{\lambda}(\|\nabla_x M\|_\infty, \kappa ) |\nabla_x M (x)|^b. 
		\end{align*}
		\[ \] 
	\end{lem}
	\begin{proof}
		Let us first prove the lemma when $d=2$ which is the application of Theorem \ref{thm:angulardep}. Performing the same change of variables as in Lemma \ref{lem:FL1}, we have
			\begin{multline*}
				\int_{\S^{1}} \kappa(v',v) v' \cdot \nabla_x M(x) \psi(v' \cdot \nabla_x M (x) ) \d v' \\
				=\int_{-\pi}^\pi \kappa_1(\theta) (\cos (\theta)v \cdot \nabla_x M + \sin (\theta) v^\perp \cdot \nabla_x M) \psi (\cos (\theta) v \cdot \nabla_x M + \sin (\theta)  v^\perp \cdot \nabla_x M) \d \theta.
			\end{multline*}
			Then, using the assumption \eqref{asm:m_psi_m} we can bound this integral below by
			\begin{align*}
				c \int_{-\pi}^\pi \kappa_1(\theta) |\cos (\theta) v \cdot \nabla_x M + \sin (\theta) v^\perp \cdot \nabla_x M |^b \d \theta = c |\nabla_x M |^b \int_{-\pi}^\pi \kappa_1(\theta) |\cos(\theta - \alpha)|^b \d\theta,
			\end{align*} where $\alpha$ is the angle between $v$ and $\nabla_x M$. 
			If we write
			\begin{align*}
				F(\alpha) = \int_{-\pi}^\pi \kappa_1(\theta) |\cos(\theta - \alpha)|^b \d \theta,
			\end{align*}
			then we can compute that 
			\begin{align*}
				F'(\alpha) = b\int_{-\pi}^{\pi} \kappa_1(\theta)|\cos^2(\theta - \alpha)|^{\frac{b}{2}-1} \cos(\theta-\alpha)\sin(\theta-\alpha) \d \theta 
			\end{align*} and by changing variables to get $F'(\alpha) = b \int_{-\pi}^\pi \kappa_1(\theta+\alpha) |\cos^2(\theta)|^{\frac{b}{2}-1} \cos(\theta)\sin(\theta) \d \theta$ and then changing variables back we have
			\begin{align*}
				F''(\alpha) = b\int_{-\pi}^\pi \kappa_1'(\theta)|\cos^2(\theta - \alpha)|^{\frac{b}{2}-1} \cos(\theta-\alpha)\sin(\theta-\alpha) \d \theta. 
			\end{align*}
			Therefore, $F(\alpha) = 0$ when $\alpha = \left \{0, \pm \frac{\pi}{2}, \pm \pi \right \}$ and in the case $\alpha = \pm \frac{\pi}{2}$ we have $F''(\alpha) < 0$ as
			\begin{align*}
				F'' \left (\pm \frac{\pi}{2} \right ) = - b \int_{-\pi}^\pi \kappa_1'(\theta) |\sin^2(\theta)|^{b/2-1} \sin(\theta)\cos(\theta) \d \theta,
			\end{align*} which is negative when $\kappa_1'(\theta) \sin(\theta) \geq 0$ for all $\theta$ which will be the case if $\kappa_1$ is a decreasing function of $|\theta|$. Therefore, for all $\alpha$ we have
			\[  F(\alpha) \geq \int_{-\pi}^\pi \kappa_1(\theta)|\sin(\theta)|^b \mathrm{d}\theta. \]
			In the case where we have $d >2$, let us make the change of variables
			\begin{align*}
				v' = \cos (\theta)  v + \sin \theta \cos(\psi) u + \sin(\theta) \sin(\psi) p,
			\end{align*} where $u$ is the unit vector in the direction $\nabla_x{M} - \left(\nabla_x M \cdot v\right) v$ and $p$ is a variable vector that ranges over the sphere of vectors of norm $1$ perpendicular to both $v$ and $u$ (which we call $S_{v,u}$. The Jacobian of this change of variable is $|\sin(\theta)|^{d-2}$. Therefore, we want to find a lower bound for the integral
			\begin{align*}
				\int_{-\pi}^\pi \int_{-\pi}^\pi \int_{S_{v,u}} \kappa_1(\theta) |\sin(\theta)|^{d-2} |\cos (\theta) v \cdot \nabla_x M + \sin(\theta) \cos(\psi) u \cdot \nabla_x M|^b \d p \d \psi \d \theta. 
			\end{align*}
			The integral $\d p$ just gives us a constant factor. To evaluate the rest, we write $\phi_v$ to be the angle between $\nabla_xM$ and $v$. Then we have
			\begin{align*} 
				|\nabla_x M|^b \int_{-\pi}^\pi \int_{-\pi}^\pi &\kappa_1(\theta) |\sin(\theta)|^{d-2} |\cos(\theta) \cos(\phi_v) +\sin(\theta) \sin(\phi_v) \cos(\psi)|^b \d \psi \d \theta\\
				& = |\nabla_xM|^b \int_{-\pi}^\pi (1-\sin^2(\phi_v) \sin^2(\psi))^{b/2}\int_{-\pi}^\pi \kappa_1(\theta)|\sin(\theta)|^{d-2}|\cos(\theta-\alpha(\phi_v, \psi))|^b \d \theta \d \psi\\
				& \geq  |\nabla_x M|^b \left( \int_{-\pi}^\pi |\cos(\psi)|^b \mathrm{d}\psi\right) \inf_{\alpha}\left( \int_{-\pi}^\pi \kappa_1(\theta)|\sin(\theta)|^{d-2} |\cos(\theta - \alpha)|^b \mathrm{d}\theta \right) \\
				& =|\nabla_x M|^b \left( \int_{-\pi}^\pi |\cos(\psi)|^b \mathrm{d}\psi\right) \left( \int_{-\pi}^\pi \kappa_1(\theta)|\sin(\theta)|^{d-2} |\sin(\theta)|^b \mathrm{d}\theta\right).
			\end{align*} In the last line above, we have used similar considerations to the case $d=2$.
	\end{proof}
	
	Now, we can move on to the proof of the Foster-Lyapunov condition \eqref{FL1}. 
	\begin{prp} \label{prp:lyapunovangle}
		If hypotheses \textbf{\em (H1), (H2), (H3)} and \textbf{\em (H4)} are satisfied then we can choose a constant $A$ so that the function 
		\begin{align*}
			\phi (x,v) = \Big ( 1 - \frac{\gamma }{1- C_\kappa}v\cdot \nabla_x M(x) - \gamma A v\cdot \nabla_x M(x) \psi(v \cdot \nabla_x M) \Big )e^{-\gamma M(x)}
		\end{align*} verifies a Foster-Lyapunov condition \eqref{FL1} for Eq. \eqref{eq:rt}.
	\end{prp}
	
	\begin{proof}
		First, we notice that if $A$ is sufficiently small then
		\begin{align*}
			\phi (x,v) \geq \Big (1- 2 \gamma \Big ( \frac{1}{1-C_\kappa} + \frac{\chi}{1+\chi} \Big ) V_0 \|\nabla_x M\|_\infty \Big )e^{-\gamma M(x)}, 
		\end{align*} so we can choose
		\begin{align*}
			\gamma \leq \frac{1}{4V_0 \|\nabla_x M\|_\infty}\Big ( \frac{1}{1-C_k} + \frac{\chi}{1+\chi} \Big )^{-1} 
		\end{align*} and this ensures
		\begin{align} \label{FL_phi}
			\phi (x,v) \geq \frac{1}{2} e^{-\gamma M(x)}.  
		\end{align}
		Now, we differentiate the separate parts of the Lyapunov function, remembering \eqref{eq:L_adjoint} and using Lemmas \ref{lem:FL1} and \ref{lem:FL12}, 
		\begin{align*}
			\mathcal{L}^*( e^{-\gamma M})= -\gamma v \cdot \nabla_x M e^{-\gamma M},
		\end{align*}
		\begin{multline*}
			\mathcal{L}^* (v \cdot \nabla_x M e^{-\gamma M}) = v^T \hess (M) v e^{-\gamma M} + \gamma(v \cdot \nabla_x M)^2 e^{-\gamma M} \\
			\quad - (1-C_\kappa) v\cdot \nabla_x M e^{-\gamma M} + \chi(1-C_\kappa) \psi(v \cdot \nabla_x M) v\cdot \nabla_x M e^{-\gamma M}, 
		\end{multline*} and 
		\begin{align*}
			\mathcal{L}^*( v \cdot \nabla_x M \psi (v \cdot \nabla_x M)&e^{-\lambda M}) \geq v^T \hess(M) v \left  ( \psi(v \cdot \nabla_x M) + \psi'(v \cdot \nabla_x M)v\cdot \nabla_x M \right ) e^{- \gamma M} \\
			& \quad +\gamma (v \cdot \nabla_x M)^2 \psi(v \cdot \nabla_x M)e^{-\gamma M} - v\cdot \nabla_x M \psi(v \cdot \nabla_x M)(1+\chi)e^{-\gamma M}\\
			& \quad + \tilde{\lambda}(\|\nabla_x M\|_\infty, \kappa_1)(1- \chi) |\nabla_x M|^b.
		\end{align*}
		Now, we can put this all together in a Lyapunov functional for positive $A$ sufficiently small
		\begin{align*}
			\mathcal{L}^* \Big( \Big ( 1 - \frac{\gamma}{1- C_\kappa}
			v\cdot &\nabla_x M - \gamma A v\cdot \nabla_x M \psi(v \cdot \nabla_x M) \Big )e^{-\gamma M} \Big ) \\  \leq  & \, \gamma v^T \hess(M) v e^{-\gamma M} \Big (\frac{1}{1-C_\kappa} + A \Big (1+ \sup_z (z \psi'(z))) \Big ) + \gamma^2 V_0^2 e^{-\gamma M} \Big ( \frac{1}{1-C_\kappa}+A \Big ) \\
			& + v \cdot \nabla_x M \psi(v \cdot \nabla_x M) e^{-\gamma M} \gamma ( -\chi + A (1+ \chi)) -  \gamma A(1-\chi)|\nabla_x M|^b e^{-\gamma M}
		\end{align*}
		Then, choosing $A \leq \frac{\chi}{1+\chi}$ we have
		\begin{multline*}
			\mathcal{L}^* \Big( \Big ( 1 - \frac{\gamma}{1-C_\kappa} v\cdot \nabla_x M - \gamma  A v\cdot \nabla_x M \psi(v \cdot \nabla_x M \Big )e^{-\gamma M} \Big)  \\ 
			\leq \gamma \Big ( C_1V_0^2 |\hess(M)| e^{-\gamma M}  + \gamma C_2 V_0^2 e^{-\gamma M}  - A(1-\chi) |\nabla_x M|^b \Big ) e^{-\gamma M}.
		\end{multline*}
		Now, we know that we can choose $R_*$ such that if $|x|>R_*$ then,
		\begin{align*}
			|\nabla_x M(x)| \geq m_*\quad \mbox{and } \quad C_1 V_0^2 |\hess(M)(x)| \leq A(1-\chi) \frac{m_*^{b}}{4},
		\end{align*}then if we choose $\gamma$ small enough so that $\gamma C_2  V_0^2 \leq A(1-\chi) \dfrac{m_*^{b}}{4}$, then we will have 
		\begin{multline}
			\mathcal{L}^* \Big( \Big ( 1 - \frac{\gamma}{1-C_\kappa} v\cdot \nabla_x M -\gamma  A v\cdot \nabla_x M \psi(v \cdot \nabla_x M \Big )e^{-\gamma M} \Big)  \\ 
			\leq  \gamma \Big ( -\frac{1}{2}A(1-\chi)m_*^b e^{-\gamma M} \1_{|x| > R_*} + C_3 \1_{|x| \leq R_*} \Big ),
		\end{multline} for some constant $C_3$. Finally, using \eqref{FL_phi} we have
		\begin{align*}
			\mathcal{L}^*[\phi] \leq -\gamma A (1-\chi)m_*^b \phi  + C_4. 
		\end{align*} This verifies \eqref{FL1} with $\zeta = \gamma A (1-\chi) m_*^b$ and $D = C_4$. 
	\end{proof}
	
	\begin{proof}[Proof of Theorem \ref{thm:angulardep}] We verify the two hypotheses of Harris's theorem in Propositions \ref{prop:minorisation_angle_dependent} and \ref{prp:lyapunovangle}. The contraction in the $\phi-$weighted total variation norm and the existence of a unique steady state follow by Harris's theorem (Theorem \ref{thm:harris}). 
	\end{proof}
	
	\section{Unbounded velocity spaces}
	\label{sec:unbounded-velocities}
	
	This section is dedicated to the long-time behaviour of the linear run and tumble equation, Eq. \eqref{eq:rt} posed in an unbounded velocity space $\mc V = \R^d$ with a tumbling kernel \eqref{eq:K_Maxwell}. 
	The following two sections are dedicated to verifying the two hypotheses of the subgeometric version of Harris's theorem (Theorem \eqref{thm:subgeo}). At the end of this section, we provide a proof of Theorem \ref{thm:unbounded}.
	
	\subsection{Minorisation condition}
	
	For unbounded velocities, our minorisation part is essentially identical to that in our previous paper \cite{EY21}. For the sake of completeness, we include a proof here. Let us again write $(T_t f) (t,x,v) = f(t,x- vt, v)$ for the transport semigroup and define
	\begin{align*}
		(\hat{L} f) (t,x,v) := \int_{\R^d} f(t, x,u) \d u \1_{|v| \leq V_0} (v).
	\end{align*}
	Then we have the following lemmas.
	\begin{lem}
		There exists a constant $C$ such that 
		\[ f(t,x,v) \geq C e^{-(1+\chi)t} \int_0^t \int_0^s T_{t-s}\hat{L} T_{s-r} \hat{L} T_r f_0 (x,v) \d r \d s. \]
	\end{lem}
	\begin{proof}
		The proof is exactly the same as in \cite{EY21} after observing that there exists some $\tilde C>0$ such that 
		\begin{align*}
			\kappa_2(v) \geq \tilde C \1_{|v| \leq V_0} (v).
		\end{align*}
	\end{proof}
	\begin{lem} \label{lem:M2}
		For every $R_*>0$, we can take $t = 3 +\frac{R_*}{V_0}$ such that any solution of Eq. \eqref{eq:rt} with initial data $f_0 \in \mathcal{P}(\R^d \times \R^d)$ with $\int_{|x|\leq R_*} \int_{B(0; V_0)}f_0(x,v) \d x \d v =1$ satisfies 
		\begin{align}
			f(t, x,v) \geq (1-\chi^2) e^{-(1+\chi)t} \frac{1}{t^d |B(V_0)|} \1_{|x| \leq V_0}(x)\1_{|v| \leq V_0} (v).
		\end{align}
	\end{lem}
	\begin{proof}
		We take $f_0 (x,v): = \delta _{(x_0} \delta_{v_0)} $ where $(x_0,v_0) \in\R^d \times B(0,V_0)$, is an arbitrary point with an arbitrary velocity. We only need to consider $x_0 \in B(0, R_*)$, then the bound we obtain depends on $R_*$.
		First, we have that 
		\begin{align*}
			T_rf_0 \geq \delta_{x_0 + rv_0} (x)\delta_{v_0}(v).
		\end{align*}
		Applying $\hat{L}$ to this we obtain
		\begin{align*}
			\hat{L} T_r f_0 \geq \delta_{x_0+rv_0}(x) \1_{|v| \leq V_0} (v).
		\end{align*}
		Performing a change of variables we have 
		\begin{align*}
			\int_{\R^d}  (T_{s-r} \hat{L}  \mathcal T_r f_0 )(x,v) \d v \geq   \frac{1}{(s-r)^d |B(V_0)|} \1_{|x-x_0-rv_0| \leq V_0 (s-r)} (x).
		\end{align*} Therefore we have
		\begin{align*}
			\hat{L} T_{s-r}\hat{L}T_r f_0 \geq    \frac{1}{(s-r)^d |B(V_0)|}  \1_{|x-x_0-rv_0| \leq V_0 (s-r)}  (x) \1_{|v| \leq V_0} (v).
		\end{align*} 
		Applying the transport operator once more we obtain
		\begin{align*}
			T_{t-s} \hat{L}T_{s-r}\hat{L}T_r f_0 \geq  \frac{1}{(s-r)^d |B(V_0)|}  \1_{|x-(t-s)v-x_0-rv_0| \leq V_0 (s-r)}  (x) \mathds{1}_{|v| \leq V_0}(v).
		\end{align*} Moreover, we have
		\begin{align*}
			|x| \leq  (s-r) V_0 - (t-s)V_0 - r V_0 - R_*
		\end{align*} which is implied by $|x-v(t-s)-x_0 -rv_0| \leq (s-r) V_0$ since all the velocities are smaller than $V_0$.
		Then, if we ensure that $(s-r) \geq 2 + \frac{R_*}{V_0}$, $ r \leq \frac{1}{2}$ and that $(t-s) \leq \frac{1}{2}$, we will obtain
		\begin{align*}
			T_{t-s} \hat{L}T_{s-r}\hat{L}T_r f_0 \geq \frac{1}{(s-r)^d |B(V_0)|}  \1_{|x| \leq V_0}  (x)\1_{|v| \leq V_0} (v).
		\end{align*}
		Therefore, setting $t = 3+\frac{R_*}{V_0}$ and restricting the time integrals to $r \in \big (0,\frac{1}{2} \big)$, $s \in \big (\frac{5}{2}+ \frac{R_*}{V_0}, 3 + \frac{R_*}{V_0} \big )$, we obtain
		\begin{align*}
			f(t,x,v)&\geq C\int_{0}^{t} \int_{0}^{s}T_{t-s} \hat{L}T_{s-r}\hat{L}T_r f_0 (x,v)\d r \d s \\
			&\geq C(1-\chi)^2 e^{-(1+\chi)t}  \int_{\frac{5}{2} + \frac{R_*}{V_0}}^{3 + \frac{R_*}{V_0}} \int_{0}^{\frac{1}{2}}\frac{1}{(s-r)^d |B(V_0)|} \1_{|x| \leq V_0} (x)  \1_{|v| \leq V_0} (v)\d r \d s\\
			&\geq  C(1-\chi)^2 e^{-(1+\chi)t}   \frac{1}{t^d |B(V_0)|} \1_{|x| \leq V_0} (x) \1_{|v| \leq V_0} (v).
		\end{align*}
		This gives the uniform lower bound and verifies the minorisation condition \eqref{M2}. We can extend this result from the Dirac delta function initial data to general initial data by using the fact that the associated semigroup is Markov. 
	\end{proof}

	\subsection{Foster-Lyapunov condition}
	\label{sec:Foster-Lyapunov2}
	In this section, we verify the Foster-Lyapunov condition \eqref{FL2} for Eq. \eqref{eq:rt} for the case \eqref{eq:K_Maxwell}. Thus, we prove the following lemma. 
	\begin{lem} If hypotheses \textbf{\em (H1), (H2), (H3), (H5)} are satisfied and assuming that $M \hess (M)$ is bounded, then the function 
		\begin{align*}
			\phi (v, M) = M^2 + 2 v \cdot \nabla_x M M \Big( 1+ \frac{\chi}{1+\chi}\psi(v \cdot \nabla_x M)\Big) + Av^2 
		\end{align*} verifies a weaker Foster-Lyapunov function \eqref{FL2} for a constant $A$ sufficiently large and with computable constants $C>0, \Lambda >0$ so that
		\begin{align*}
			\mathcal{L}^* [\phi](v,M) \leq C - \Lambda \sqrt{\phi(v,M)}. 
		\end{align*}
	\end{lem}
	\begin{proof}
	In this proof, it is useful to remember that $M$ is negative for $|x|$ sufficiently large. Since it is only defined up to a constant, let us choose $M < 0$. Similar to the previous case, we look at how the adjoint $\mc L^*$ (defined by \eqref{eq:L_adjoint}) acts on different terms. Precisely, we have
			\begin{align*}
				\mathcal{L}^*(M^2)  &= 2v \cdot \nabla_x MM, \\
				\mathcal{L}^*(2v\cdot\nabla_x M M) &= 2 v^T \mbox{Hess}(M) v M + 2 (v \cdot \nabla_x M)^2 -2v \cdot \nabla_x M M - 2 \chi v \cdot \nabla_x M \psi(v \cdot \nabla_x M) |M|
			\end{align*}
			Then we have, for any $c>0$
			\begin{align*}
				\mathcal{L}^* \big (c v \cdot \nabla_x M & \psi(v \cdot \nabla_x M ) M \big )
				\\ = \, &c v^T \hess(M)v M \left (\psi'(v \cdot \nabla_x M) v\cdot \nabla_x M + \psi(v \cdot \nabla_x M) \right )  + c(v \cdot \nabla_x M)^2 \psi(v \cdot \nabla_x M) \\
				& + c(1-\chi\psi(v \cdot \nabla_x M)) M  \left ( \int \kappa_2(v')  v' \cdot \nabla_x M \psi(v' \cdot \nabla_x M) \d v' - v \cdot \nabla_x M \psi(v \cdot \nabla_x M ) \right ) \\
				\leq & \, c \left( \|\hess(M) M\|_{\infty} \| \psi(z)(z)\|_{\text{Lip}} + \|\nabla_x M\|^2_{\infty}\right) |v|^2 \\
				&  - c(1-\chi) \tilde{\lambda}(\|\nabla_x M\|_\infty, \kappa_2)|M| + c (1+\chi) v \cdot \nabla_x M \psi(v \cdot \nabla_x M) |M|.
			\end{align*}
			Summing these up and choosing $c = \frac{2\chi}{1+\chi}$ we have
			\begin{multline*}
				\mathcal{L}^* \big( M^2 +2 v \cdot \nabla_x M M + \frac{2\chi}{1+\chi} v \cdot \nabla_x M \psi(v \cdot \nabla_x M)M \big)  
				\\ \leq \Big( \Big (2 + \frac{2\chi}{1+\chi}\|\psi(z)(z)\|_{\text{Lip}} \Big) \|\hess(M)M\|_{\infty} + \Big(2 + \frac{2\chi}{1+\chi} \Big )\|\nabla_x M\|_{\infty}^2 \Big  )|v|^2  - \frac{2\chi(1-\chi) }{1+\chi} \tilde{\lambda}|M|
			\end{multline*}
			Now we also have
			\begin{align*}
				\mathcal{L}^* (|v|^2) \leq (1+\chi) - (1-\chi)|v|^2. 
			\end{align*}
			Therefore choosing 
			\begin{align*}
				A \geq 1 + \frac{1}{1-\chi} \Big( \Big(2 + \frac{2\chi}{1+\chi}\|\psi(z)(z)\|_{\text{Lip}} \Big) \|\hess(M) M\|_\infty + \Big(2 + \frac{2\chi}{1+\chi} \Big )|\nabla_x M\|^2_{\infty} \Big),
			\end{align*} we have
			\begin{align*}
				\mathcal{L}^* \Big ( M^2 +2 v \cdot \nabla_x M M + \frac{2\chi}{1+\chi} v \cdot \nabla_x M\psi(v \cdot \nabla_x M)&M + A |v|^2 \Big )
				\\ & \leq A(1+\chi) - (1-\chi)|v|^2  - \frac{2\chi(1-\chi) \tilde{\lambda}}{1+\chi} |M|\\
				& \leq C - \Lambda \sqrt{\phi(v,M)}.
			\end{align*}
			Lastly, choosing $A$ sufficiently large ensures that $\phi >0$, and for $A$ sufficiently large, $\phi$ is comparable to $M^2+v^2$. 
	\end{proof}
	\subsection{Subgeometric convergence rates}
	\label{sec:subgeo}
	We can now combine the results of the two previous sections to get a proof of Theorem \ref{thm:unbounded}.
	\begin{proof}[Proof of Theorem \ref{thm:unbounded}]
		We have verified the hypotheses of the subgeometric Harris's theorem with the Foster-Laypunov function being 
		\begin{align*}
			\phi(v, M) = M^2 + 2 v \cdot \nabla_x M M \Big( 1+ \frac{\chi}{1+\chi}\psi(v \cdot \nabla_x M)\Big) + Av^2, 
		\end{align*} and the function $h(t) = \sqrt{t}$. We can, therefore, compute that the function $H^{-1}(t) = \big (\frac{t}{2}+1 \big )^2$ and $h \circ H_h^{-1}(t) = \big (\frac{t}{2} +1 \big )$. Hence, applying the conclusion of Theorem \ref{thm:subgeo} gives the existence of a steady state $f_\infty$ and that
		\begin{multline*}
			\|f_t - f_\infty\|_{TV} \\ \leq C \Big (\frac{t}{2} +1 \Big )^{-2}  \int f_0 \left (M^2 + 2 v \cdot \nabla_x M M \Big( 1+ \frac{\chi}{1+\chi}\psi(v \cdot \nabla_x M)\Big)+ Av^2 \right ) \d x \d v  + C \Big (\frac{t}{2} +1 \Big )^{-1}. 
		\end{multline*} This proves the result.
	\end{proof}

	\section*{Acknowledgements}
	The authors would like to thank Emeric Bouin for helpful discussions and for suggesting to look at unbounded velocity spaces. 
	
	The authors would like to thank the Isaac Newton Institute for Mathematical Sciences for support and hospitality during the programme ``Frontiers in kinetic theory: connecting microscopic to macroscopic scales - KineCon 2022'' when work on this paper was undertaken. This work was supported by EPSRC Grant Number EP/R014604/1. 
	
	J. Evans acknowledges partial support from the Leverhulme Trust, Grant ECF-2021-134.
	H. Yolda\c{s} was partially supported by the Vienna Science and Technology Fund (WWTF) with a Vienna Research Groups for Young Investigators project, grant VRG17-014 (until October 2022).

	\bibliography{Run-and-Tumble}

\end{document}